\documentclass[11pt]{amsart}

\usepackage{amsmath}
\usepackage{amssymb}
\usepackage{graphicx}
\usepackage{xcolor}

\newtheorem{theorem}{Theorem}[section]

\newtheorem{corollary}[theorem]{Corollary}
\newtheorem{prop}[theorem]{Proposition}

\theoremstyle{definition}
\newtheorem{definition}[theorem]{Definition}
\newtheorem{example}[theorem]{Example}

\theoremstyle{remark}
\newtheorem{remark}[theorem]{Remark}

\newcommand{\diam}{\text{diam}}

\numberwithin{equation}{section}

\keywords{Metric spaces, metric normal structure, orbits, nonexpansive operators,  convexity structure, action of groups,  common fixed points }
\subjclass[2010]{47H09, 47H10, 47H20, 54E40}
\begin{document}

\title[ Fixed points for orbit-nonexpansive mappings]{Fixed points and common fixed points for orbit-nonexpansive mappings in metric spaces}
\author{Rafael Esp\'{\i}nola, Maria Jap\'on, Daniel Souza}

\address{R. Esp\'{\i}nola\hfill\break
Departamento de An\'{a}lisis Matem\'{a}tico, Universidad de
Sevilla, Spain}
 \email{espinola@us.es}

 \address{M. Jap\'on \hfill\break
Departamento de An\'{a}lisis Matem\'{a}tico, Universidad de
Sevilla, Spain}
\email{japon@us.es}
\address{D. Souza \hfill\break
Departamento de An\'{a}lisis Matem\'{a}tico, Universidad de
Sevilla, Spain}
 \email{jdsouzaufrj@gmail.com}

\thanks{The authors are partially supported by Ministerio de Ciencia e Innovaci\'on, Grant  PGC2018-098474-B-C21, and Junta de Andaluc\'{\i}a Grants P20-00637, US-1380969 and
FQM-127.}

\begin{abstract}
In this paper we introduce an interlacing condition on the elements of a family of operators that allows us to gather together a number of results on fixed points and common fixed points for single and families of mappings defined on metric spaces. The innovative concept studied here 
deals with nonexpansivity conditions with respect to orbits  and under assumptions that only depend on the features of the closed balls of the metric space. 
\end{abstract}

\maketitle

\section{Introduction}

Given $(M,d)$ a metric space,  a mapping $T:M\to M$ is said to be nonexpansive if 
$$
d(T x,Ty)\le d(x,y)
$$
for all $x,y\in M$. 
Nonexpansive operators arose as a forthright generalization of contractive mappings introduced by S. Banach. Easy examples show that nonexpansive self-mappings defined on a complete metric space may fail to have a fixed point. However, in the framework of Banach spaces,  W. A. Kirk  proved  the existence of fixed points for every nonexpansive self-mapping defined from a convex weakly compact subset  under the additional assumption of normal structure  \cite{K}. The notion of normal structure was defined by Brodski\u{i} and Milman in \cite{BM}  as a key tool to show the existence  of a common fixed point for the group of (onto) isometries defined on a convex weakly compact subset of a Banach space.
Since then, a  vast literature has appeared connecting  geometry properties of Banach spaces with the   existence of fixed points for nonexpansive mappings as well as with the existence of common fixed points for families of operators satisfying nonexpansivity and  some extra requirement.

As it is the case for contractions and isometries, being nonexpansive for a mapping is a pure metric feature that does not call for the definition of any underlying linear structure: it only depends on the relationship between the distance of points and the distance of their images. Therefore, it is not surprising that many authors have studied some types of extensions when it comes to  the existence of fixed points for nonexpansive operators   on complete metric spaces (see, for instance, \cite{EFP,EJS,EK,Kh0, Kh,KK,La,N} and references therein).


 \medskip

 In this paper, our general framework will be a metric space and our main purpose is the analysis of the existence of  (common) fixed points for ope\-ra\-tors that verify a nonexpansivity condition with respect to  orbits and for which continuity assumptions are not needed. These operators were firstly introduced by A. Nicolae  \cite{N} (see also \cite{AFH})
 and  include nonexpansive mappings as a particular case.  The precise definitions will be given  in the next section. A similar notion requiring continuity was initially considered in \cite{KS}
in relation to mappings with diminishing orbital diameters. 
  More recently, orbit conditions have been considered in the study of existence of best proximity points for cyclic and non cyclic mappings \cite{DEK, GM}.

  With tools  relying only upon  features  of the closed balls of the metric space,
we will be able to extend some previous results in the Banach framework for orbit-nonexpansive mappings and, additionally,  in the  much wider setting of metric spaces.  As a particular case, we will deduce the common fixed point result for bounded  hyperconvex metric spaces given  in \cite{Kh,La} as well as those obtained for CAT(0) metric spaces in \cite{N}. 
New directions beyond the prior scopes will also be provided throughout the article.
  
  \medskip

The organization of the paper is the following: In Section 2 we will establish some basic definitions as well as the metric framework on which we are going to work. We will introduce the notion of orbit-nonexpansive mapping and  we will  provide examples of such mappings emphasizing the lack of  continuity.
 In Section 3 we will focus on the existence of fixed points for orbit-nonexpansive operators. With the purpose of ana\-ly\-zing  simultaneously 
 the case of a single mapping and  the action of a  group, we will introduce an {\it interlaced} condition for a family of operators defined on a metric space, upon which our  main theorems will be proved.   Although they do not generally form a group, following similar techniques as in \cite{Kh},  the existence of a  common fixed point for a family of commuting operators will also be achieved. The concept of {\it metric} normal structure will be essential in our approach. Additionally, 
 we will show how our results  generalize \cite[Theorem 2.2]{AFH}, \cite[Theorem 5.1]{N}, \cite[Theorem 8]{Kh} and some others, as well as they can be applied to other environments which are not covered by the previous cited papers.

 In Section 4, we will focus on the condition of uniform relative normal structure, initially defined by Soardi \cite{S} in the linear case, for obtaining fixed points for nonexpansive operators defined in abstract $M$ Banach lattices, and also used by T. Lau \cite{TLau} for the action of isometries. Once more, our framework will be a general metric space as in \cite{EJS, KK} and the interlaced condition will be used to obtain fixed points for a single mapping and for the action of a group likewise.




\section{Preliminaries} 

Throughout this paper, $(M,d)$ will denote a metric space. Given a mapping $T:M\to M$ and $x\in M$, the orbit of $x$ under $T$ is denoted by 
$$
o_T(x):=\{x\}\cup \{T^n x: n\in\mathbb{N}\}.
$$
Given a bounded subset  $A$ of $M$, we denote by $\delta(A)$ its diameter, that is, $\delta(A):=\sup\{d(x,y):x,y\in A\}$. Additionally, for $x\in M$ we set
$$
D(x,A):=\sup\{d(x,a): a\in A\},
$$
which gives the least radius of a closed ball centered at $x$ containing $A$. (Note that $D(x,A)$ is also denoted by $r_x(A)$ in some other references). The following definition was initially introduced in \cite{N}:

 \begin{definition}\label{def1} Given a metric space $(M,d)$, a mapping $T:M\to M$ is said to be orbit-nonexpansive  (or nonexpansive with respect to orbits) if 
\begin{equation}\label{orbit}
	d(Tx,Ty)\le D(x, o_T(y))\qquad \forall x,y\in M.
\end{equation}
\end{definition}

It is clear that every nonexpansive mapping satisfies (\ref{orbit}) since $y\in o_T(y)$ for all $y\in M$. On the other hand, it is not difficult to check that condition (\ref{orbit}) further implies: $d(T^nx,T^ny)\le D(x, o_T(y))$ for all $x,y\in M$ and $n\in\mathbb{N}$. In the event that $(M,d)$ is an unbounded metric space,  the following disjunctive follows: either all orbits of $T$ are bounded or all orbits  of $T$ are unbounded.  If a mapping $T$ verifies that all its orbits are unbounded, condition (\ref{orbit}) trivially holds and no fixed point can be assured as a simple translation in $\mathbb{R}$ shows. We will exhibit later an example that no fixed point can be guaranteed even if all of T's orbits are bounded in an unbounded metric space
 (see Remark \ref{prus}). Thus, from this moment on we will assume that the metric space $(M,d)$ is bounded. 

\medskip

 A. Nicolae   proved that  every self orbit-nonexpansive operator has a fixed point when it is defined on a bounded complete CAT(0) space (and therefore on a bounded closed and convex subset of a Hilbert space) \cite[Theorem 5.1]{N}. 
In a later paper, A. Amini-Harandi, M. Fakhar and H. R. Hajisharifi \cite{AFH} studied the existence of fixed points for orbit-nonexpansive mappings  in the Banach space setting under weak normal structure. In fact, 
 \cite[Theorem 2.2]{AFH}  proves that a Banach space $X$ has weak normal structure if and only if for every convex weakly compact subset $C$ and for every $T:C\to C$ orbit-nonexpansive mapping there exists a fixed point. It should be highlighted  that weak normal structure in Banach spaces had previously been characterized by using fixed points for the so-called Jaggi nonexpansive mappings (see \cite{J,K}). Nevertheless the definition of this class of  mappings introduced by Jaggi   in \cite{J}   strongly uses the linear structure of the Banach space.



\medskip

It is clear that nonexpansivity implies uniform continuity. However, being orbit-nonexpanse does not require continuity as the  example given in  \cite[Example 5.2]{N} shows.
We next display  a further example of an orbit-nonexpansive mapping  for which it is not possible to find a non-trivial closed $T$-invariant interval  upon which the mapping  $T$  is continuous.

\begin{example}\label{e2} Let $T:[-1,1]\to [-1,1]$ given by
$$
T(x)=\left\{
\begin{array}{rll}
{x\over 3}& \mbox{if} & \mbox{$x$ is irrational},\\
-{x\over 3} & \mbox{if} & \mbox{$x$ is rational}.\\
\end{array}\right.
$$
We prove that $T$ is orbit-nonexpansive. Notice that $0\in \overline{o_T(x)}$ for all $x\in [-1,1]$. Take first $x,y\ge 0$ such that $x$ is irrational and $y$ rational. We split the proof into cases: 

Case 1: $0\le x<y\le1$. In this case $\displaystyle D(x,o_T(y))=\max\{y-x,x+\frac{y}{3}\}$ and $\displaystyle D(y,o_T(x))=y$. Moreover,
$$
d(Tx,Ty)={x\over 3}+{y\over 3}\le x+{y\over 3}\le D(x,o_T(y)) \ \text{and}
$$

$$
d(Ty,Tx)={x\over 3}+{y\over 3}\le 2{y\over 3}\le D(y,o_T(x)).
$$


Case 2: $0\le y<x\le1$. In this case $\displaystyle D(x,o_T(y))=x+\frac{y}{3}$ and $\displaystyle D(y,o_T(x))=\max\{y,x-y\}$. Moreover,
$$
d(Tx,Ty)={x\over 3}+{y\over 3}\le {2x\over 3}\le x\le D(x,o_T(y)).
$$
Now, if $\displaystyle y\le {x\over 2}$,
$$
d(Ty,Tx)={x\over 3}+{y\over 3}\le x-y\le D(y,o_T(x)).
$$
and if $\displaystyle {x\over 2}\le y$,
$$
d(Ty,Tx)={x\over 3}+{y\over 3}\le {2y\over 3}+{y\over 3}= y\le D(y,o_T(x)).
$$

By symmetry, we can conclude that $d(Tx,Ty)\le D(x,o_T(y))$ for all $x,y\in [-1,1]$ and $T$ is orbit-nonexpansive. Additionally,  $T$ fails to be nonexpansive on any $T$-invariant nontrivial closed interval of $[-1,1]$, what easily follows from the fact that $T$ is discontinuous everywhere except for the origin, which is fixed for $T$.


\end{example}

\begin{remark}
	Notice that every orbit-nonexpansive mapping is continuous at least at  the set of fixed points (provided it is not empty). This is due to the inequality
	 $d(Tx,y)\le d(x,y)$ for all $x\in M$ and for  $y\in M$ verifying  $T(y)=y$. 
	
\end{remark}


Mappings satisfying $d(Tx,Ty)\le \max\{d(x,y), d(x, Ty)\}$ for all $x,y\in M$  clearly lie  within the scope of Definition \ref{def1}.
Another family of mappings that fulfills the  orbit-nonexpansivity  is the class of mean nonexpansive mappings introduced in \cite{Z} as follows: 
A mapping $T:M\to M$ is said to be mean nonexpansive if there exist nonnegative constants $a,b\ge 0$ with $a+b\le 1$ and such that 
$$
d(Tx,Ty)\le ad(x,y)+ bd(x, Ty)
$$
for all $x,y\in M$ (see \cite{WZ,Z,Zu} and references therein).

\medskip

Since we are also interested in obtaining common fixed points for the action of a group, we  recall some standard notions: 
A set $\mathcal S$ is a group if there is an inner product, $\cdot$, defined on it such that it is associative,  there is an element $1\in \mathcal S$ with $s\cdot 1=1\cdot s=s$ for all $s\in \mathcal S$, and for every $s\in \mathcal S$ there is $s^{-1}\in \mathcal S$ with $s\cdot s^{-1}=s^{-1}\cdot s=1$. 
Given a set $M$ and a group $\mathcal{S}$, it is said that $\mathcal{S}$ acts over $M$ if every $s\in \mathcal{S}$ defines an operator from $M$ into itself and the following two properties are satisfied: the unit element $1\in \mathcal{S}$ defines the identity over $M$  and $(s\cdot t)(x)=s(tx)$ for all $s,t\in\mathcal{S}$ and $x\in M$.






 \medskip
 
Next, we recall some  facts about convexity structures in metric spaces. 

\begin{definition}

\begin{itemize}

\item[i)]
 A subset of a metric space $(M,d)$ is said to be {\it admissible} if it is an intersection of closed balls of $M$. The family of all admissible sets of $M$ is denoted by $\mathcal{A}(M)$.

\item[ii)] A family ${\mathcal F}$ of subsets of $M$ defines a {\it convexity structure} on $M$ if it contains all closed balls of $M$ and is closed under intersection.

\item[iii)] A convexity structure on $M$ is said to be {\it compact} if the intersection of any collection of elements of it is nonempty provided finite intersections of such elements are nonempty.

\end{itemize}
\end{definition}

In fact, it is known that compactness of the family of admissible sets implies completeness of the metric space \cite[Proposition 3.2]{EK}. 
It follows from the definition that the family $\mathcal{A}(M)$ of admissible sets  is  a convexity structure  contained in every other  convexity structure considered on $M$.
 
 In the setting  of  a Banach space endowed with a topology $\tau$ for which the closed balls are $\tau$-closed, the family of all convex $\tau$-closed bounded sets   forms a convexity structure.  Standard examples are Banach spaces endowed with the weak topology, 
  the weak$^*$-topology (in case of a dual Banach space) or the closed in measure topology in case of the Lebesgue space $L_1[0,1]$. In case that a  $\tau$-compact  convex subset $M$ is considered for any of the previous topologies, the resulting convexity structure formed by the convex $\tau$-closed subsets of $M$ is also compact.  These
 convexity structures turned out to  be essential tools in order to ensure the existence of fixed points for nonexpansive operators defined on convex subsets which are weakly compact, weakly$^*$  compact or compact in measure, respectively (see, for instance, \cite{BDJ, DJ, DGJ} and references therein). 


\medskip

Abstract convexity structures in metric spaces have been extensively stu\-died in the last decades. A complete monograph on related topics was given by M. A. Khamsi and W. A. Kirk in \cite{KK}, where Chapter 5 is   devoted to normal structure in the metric space setting.

\begin{definition}\label{NS} A metric space $(M,d)$ is said to have
 {\it metric} normal structure if for every non-singleton admissible set $A\in\mathcal{A}(M)$, there exists some $z_A\in A$ such that $D(z_A,A)<\delta(A)$.\end{definition} 
 
Note that the previous definition only involves the closed balls of the me\-tric space $(M,d)$. It says that for every admissible set $A$ with positive diameter, there exists a  closed ball containing $A$, with radius strictly less than the diameter of $A$ and whose center  $z_A$ is  in $A$ (in the last section we will consider a modified version of the metric normal structure where the center $z_A$ of the ball   containing $A$ is allowed to lie outside of $A$).

\medskip

When
the family of admissible sets $\mathcal{A}(M)$ is replaced in Definition \ref{NS} by some convexity structure as the ones introduced above for Banach spaces, 
the concepts of weak normal structure, weak$^*$ normal structure or $\tau$-normal structure (when $\tau$ is the topology of the convergence in measure in $L_1[0,1]$) are just displayed. In fact, it was proved in \cite{Le} that convex compact in measure subsets of $L_1[0,1]$ have $\tau$-normal structure.  
 Actually,  it is possible to find closed convex bounded subsets of a Banach space failing to have   weak normal structure (when the convexity structure is the family of all convex weakly compact subsets) but they still  have metric normal structure   \cite{EJS,lennard}. 

\medskip

Strongly connected with Definition \ref{NS} is the notion of {\it uniform normal structure} (UNS).

\begin{definition}
 A metric space $(M,d)$ is said to have UNS if there exists some $c\in (0,1)$ such that for every non-singleton admissible set $A$ there exists some $x_A\in A$ with $D(x_A,A)<c\ \delta(A)$. 
 
 \end{definition}
 
 \begin{remark}
It is known that, in complete metric spaces,  uniform normal structure implies compactness of  the family of admissible sets $\mathcal{A}(M)$ (see \cite{Kh0} or \cite[Theorem 5.4]{KK}).
\end{remark}

\medskip

We also recall the definition of a one-local retract of a metric space, which will be  useful  in the next section to extend a fixed point result obtained for a single mapping to a common fixed point result for a family of commuting operators. More details about one-local retracts can be found in \cite{Kh}, where it is shown  that they can be considered as a generalization of nonexpansive retracts, enjoying some structural properties of great interest.

\begin{definition}\label{1local}
	A subset $D$ of a metric space $(M,d)$ is a one-local retract of $M$ if, for any family of closed balls centered at $D$ with nonempty intersection, it is the case that the intersection of all of them must have nonempty intersection with $D$.
\end{definition}


\medskip

We finish this section by
stating some fixed point theorems that will be extended in what follows: 
 
\begin{theorem}\label{1}\cite[Theorem 8]{Kh}
Let $(M,d)$ be a bounded metric space with metric normal structure and such that $\mathcal{A}(M)$ is compact. Then any commuting family of nonexpansive self-mappings on $M$ 
has a common fixed point. Moreover, the set of common fixed points is a one-local retract of $M$.
\end{theorem}

 \begin{theorem}\label{2} \cite[Theorem 2.2]{ AFH} Let $M$ be a convex weakly compact subset of a Banach space $X$ with weak normal structure and $T:M\to M$ an orbit-nonexpansive mapping. Then $T$ has a fixed point. 
 \end{theorem}
 
 \begin{theorem}\label{3}\cite[Theorem 5.1]{N} Let $(M,d)$ be a bounded complete CAT(0) metric space and $T:M\to M$ an orbit-nonexpansive mapping. Then $T$ has a fixed point. 
 
 \end{theorem}





\section{Metric normal structure and fixed points for orbit-nonexpansive mappings}

In this section we will study how the notion of metric normal structure introduced in Definition \ref{NS} leads us to achieve positive results concerning the existence of fixed points for orbit-nonexpansive mappings. 
With the purpose of dealing simultaneously  with the case of a single mapping and the case of the action by a group, we introduce the following definition:

\begin{definition}\label{interlaced}
	Let $(M,d)$ be a  bounded metric space and ${\mathcal F}$ a family of self-mappings on $M$ such that
	$$
	d(Tx,Sy)\le \sup \{ D(x,o_R(y))\colon R\in\mathcal F\} \ 
	$$
	for $T,S\in \mathcal F$ and $x,y\in M$.
	Then we say that $\mathcal F$ is a family of interlaced orbit-nonexpansive mappings. \end{definition}

It is straightforward that  a mapping $T:M\to M$ is orbit-nonexpansive  if and only if the singleton family $\mathcal{F}:=\{T\}$ satisfies Definition \ref{interlaced}. Next we prove   a similar result for a group of operators acting over a metric space each of which is orbit-nonexpansive.



\begin{prop}\label{prop:groups}
	Let $(M,d)$ be a bounded metric space and let $\mathcal S$ be a group of self-operators acting on $M$, each of one is  orbit-nonexpansive. Then  the family $\mathcal S$ is  interlaced orbit-nonexpansive.
\end{prop}

\begin{proof}
	Let $s,t$ in $\mathcal S$ be orbit-nonexpansive mappings. Then,
	$$
		d(tx,sy) = d(tx,tt^{-1}sy)\le D(x,o_t(t^{-1}sy)).
	$$
	But, $t^{-1}s\in \mathcal S$, so
	$$
	 D(x,o_t(t^{-1}sy))\le \sup \{ D(x,o_r(y))\colon r\in\mathcal S\}.
	$$
	Therefore, $\mathcal S$ is a family of interlaced orbit-nonexpansive operators. \end{proof}

  Now we present the main result of this section.



\begin{theorem}\label{normal} Let $(M,d)$ be a bounded metric space with metric normal structure and such that $\mathcal{A}(M)$ is compact. Let $\mathcal F$ be a family of interlaced orbit-nonexpansive self-mappings on $M$. Then, there exists a common fixed point to all mappings in $\mathcal F$. Moreover, the common fixed point set of $\mathcal{F}$, ${\rm Fix}(\mathcal F)$, is a one-local retract of $M$. 
	
\end{theorem}

\begin{proof}
	Consider ${\mathcal A}_{\mathcal F}(M)$ the class of nonempty admissible subsets of $M$ which are $T$-invariant for every $T\in \mathcal F$. This family is nonempty since $M\in {\mathcal A}_{\mathcal F}(M)$. Zorn's Lemma and compactness of ${\mathcal A}(M)$ imply that there exists a minimal element with respect to set inclusion in ${\mathcal A}_{\mathcal F}(M)$. Let such a set be denoted as $A_0$.
We claim that $A_0$ is a singleton and so its element is a common fixed point for all mappings in $\mathcal F$ . 

By contradiction, suppose that $\delta (A_0)>0$.  To simplify the notation, for a subset $C\subset M$ we denote
	$$
{\rm cov} (C):=  \bigcap\{ B\; :\; B \text{ is a closed ball and $C\subseteq B$}\},
$$
In other words, ${\rm cov}(C)$ is the least admissible set containing  $C$.

	Since $A_0$ is  $T$-invariant for every $T\in \mathcal F$, we have that
	$$
	\displaystyle\bigcup_{T\in \mathcal F} T(A_0)\subseteq A_0.
	$$
	Since $A_0$ is admissible,
	$$
	{\rm cov}(\displaystyle\bigcup_{T\in \mathcal F} T(A_0))\subseteq A_0.
	$$
	Moreover, for $S\in \mathcal F$,
	$$
	S({\rm cov}(\displaystyle\bigcup_{T\in \mathcal F} T(A_0)))\subseteq S(A_0)\subseteq {\rm cov}(\displaystyle\bigcup_{T\in \mathcal F} T(A_0)).
	$$
	Therefore, 
	$$
	{\rm cov}(\displaystyle\bigcup_{T\in \mathcal F} T(A_0))\in {\mathcal A}_{\mathcal F}(M).
	$$
	and  the minimality of $A_0$ implies that
	\begin{equation}\label{A0equal}
	A_0={\rm cov}(\displaystyle\bigcup_{T\in \mathcal F} T(A_0)).
	\end{equation}
	
	Now, using the metric normal structure of the space, there exist 
	$a_0\in A_0$ such that $D(a_0,A_0)<\delta(A_0)$, which implies that
	$A_0\subseteq B(a_0,r)$ where $r:=D(a_0,A_0)$ and
	$$
	a_0\in \tilde{A}_0:=A_0\cap (\bigcap_{a\in A_0}B(a,r)).
	$$
	Notice that $\tilde{A}_0$ is a nonempty admissible set  strictly contained in $A_0$ since $\delta(\tilde{A}_0)\le r<\delta(A_0)$. 	If we show that $\tilde{A}_0$ is $T$-invariant for every $T\in \mathcal F$ we will have met a contradiction with the minimality of $A_0$:
	
	 Let $T\in \mathcal F$. Since $A_0\subseteq B(a_0,r)$, we have that $d(a_0,a)\le r$ for $a\in A_0$, but, since $A_0$ is $T$-invariant, we also have that $d(a_0,T^na)\le r$ for all $n\in \mathbb N$. Therefore,
	$$
	D(a_0, o_T(a))\le r
	$$
	for any $T\in \mathcal F$. Now, the interlacing property, implies that
	$$
	d(Sa_0,Ta)\le r.
	$$
	for any $S,T\in \mathcal F$. In particular, $T(A_0)\subseteq B(Sa_0,r)$ for $T, S\in \mathcal F$, and so
	$$
	\bigcup_{T\in \mathcal F} T(A_0)\subseteq B(Sa_0,r).
	$$
	From (\ref{A0equal}),
	$$
	A_0\subseteq B(Sa_0,r),
	$$
	and so 
	$$
	Sa_0\in \tilde{A}_0=A_0\cap (\bigcap_{a\in A_0}B(a,r))  
	$$
	for any $S\in \mathcal F$. That is, $\tilde{A}_0$ is $S$-invariant for every $S\in \mathcal F$ and as a conclusion we deduce that 
	 ${\rm Fix}(\mathcal F)$, the common fixed point set of all mappings in $\mathcal F$, is nonempty.

	In order to complete the proof, we next show that ${\rm Fix}(\mathcal F)$ is a one-local retract of $M$. Indeed, let $\{B(x_i, r_i): i\in I\}$ be a family of closed balls with $x_i\in {\rm Fix}(\mathcal F)$, for each $i$, such that $\displaystyle B=\bigcap_{i\in I} B(x_i, r_i)\ne \emptyset$. 
	
	We claim that $B\cap {\rm Fix}({\mathcal F})\ne \emptyset$. In fact, since $B$ is admissible, it turns out that $\mathcal{A}(B)$ inherits compactness and $B$ has metric normal structure. Furthermore, $B$ is $T$-invariant for any $T\in \mathcal F$. Indeed, let $y\in B$, then for each $i\in I$ we have that
	$$
	d(Ty,x_i)=d(Ty, Tx_i)\le \sup\{ D(y,o_S(x_i))\colon S\in {\mathcal F}\}=d(y,x_i)\le r_i.
	$$ 
	Repeating the first part of this proof for the metric space $B$, there is some $x\in B$ such that $T(x)=x$ for every $T\in \mathcal F$, which shows that $B\cap Fix(T)\ne \emptyset$.
\end{proof}

If we put together Theorem \ref{normal} with Proposition \ref{prop:groups} we obtain the next common fixed point result for the action of a group. The result seems to be unknown even for the case of nonexpansivity.

\begin{corollary}\label{g}
	Let $(M,d)$ be a bounded metric space with metric normal structure such that ${\mathcal A}(M)$ is compact. Let ${\mathcal S}$ be a group action over $M$ formed by  orbit-nonexpansive mappings.  Then there is $x\in M$ such that $s(x)=x$ for all $s\in \mathcal{S}$. 	
\end{corollary}

\begin{remark}
Note that Corollary \ref{g} clearly extends  the common fixed point results given for onto isometries defined on a convex weakly compact subset of a Banach space with normal structure proved by Brodski\u{i} and Milman in \cite{BM} (see also \cite{HH}), and on a bounded hyperconvex metric space stated in \cite[Proposition 1.2]{La}. 
\end{remark}

When just a single mapping is considered, we can  deduce the following:

\begin{corollary}\label{coro:normal} Let $(M,d)$ be a bounded metric space with metric normal structure and such that $\mathcal{A}(M)$ is compact. Let $T:M\to M$ be  orbit-nonexpansive. Then $T$ has a fixed point. Moreover, the fixed point set of $T$, ${\rm Fix}(T)$, is a one-local retract of $M$. 
\end{corollary}

\begin{remark}\label{prus} In the approach followed along the proof of Theorem \ref{normal},  the boundedness of the metric space was used to assure that $M$ belongs to the family of admissible sets. We could raise the question whether the boundedness of the metric space could be replaced  by the boundedness of the orbits of the mapping $T$. The following example, due to S. Prus, shows that this is not possible:
Let $T:\ell_\infty\to \ell_\infty$ be given by $T(x)=(1+\limsup_n x_n, x_1,x_2,\cdots,x_n,\cdots)$. The mapping $T$ is fixed-point free, it is an isometry, all its orbits are bounded and the space $\ell_\infty$ is hyperconvex, which implies that it has uniform normal structure and then the family of the admissible sets is compact (see for instance \cite{EK}).

\end{remark}

Following the same arguments as in \cite{Kh} and  by using the fact that the set of all fixed points is a one-local  retract of $M$, we can derive the next common fixed point theorem which is a strict generalization of \cite[Theorem 8]{Kh} (see also \cite[Theorem 6.2]{EK} for the particular case of hyperconvex metric spaces).

\begin{corollary}\label{normal2}
	Let $(M,d)$ be a bounded metric space with metric normal structure such that $\mathcal{A}(M)$ is compact. Then any commutative family $(T_i)_{i\in I}$ of orbit-nonexpansive self-mappings   on $M$ has a common fixed point. Moreover, the set of the common fixed points  is a one-local retract of $M$.

\end{corollary}
\begin{proof}
	The proof of this corollary follows in the same way as those of Theorems 7 and 8 in \cite{Kh}.
	\end{proof}

Corollary \ref{normal2} extends Theorem \ref{1}, Theorem \ref{2} and Theorem \ref{3}, since for every bounded complete CAT(0) space the family of admissible sets is therefore compact (see \cite{LS} and \cite[Theorem 5.4]{KK}). Also, the more general class of uniformly convex metric spaces studied in \cite{EFP} satisfies conditions of Corollary \ref{normal2}. 
We finish this section with some applications  to the Banach space setting:

\begin{corollary}\label{banach} Let $M$ be a convex closed bounded set of a Banach space $X$ satisfying one of the following conditions:

\begin{itemize}

\item[i)] $X$ has weak normal structure and $M$ is a weakly compact set, 

\item[ii)] $X$ is a dual space with weak$^*$ normal structure and $M$ is a weak$^*$ compact set, 

\item[iii)] $X=L_1[0,1]$ and $M$ is a compact in measure set.

\end{itemize}
Then every commuting family $(T_i)_{i\in I}$ of orbit-nonexpansive self-mappings on $M$ has a common fixed point. Additionally, the set of common fixed points of this family of mappings is a one-local retract of $M$. Furthermore, if ${\mathcal S}$ is a group of orbit-nonexpansive operators  defined on $M$, there is $x\in M$ such that $s(x)=x$ for all $s\in \mathcal{S}$.

\end{corollary}

\section{Uniform relative normal structure and fixed points for orbit-nonexpansive mappings}

P. Soardi defined in \cite{S} a geometric property in Banach spaces related to the normal structure and providing  fixed points for nonexpansive operators:  the uniform relative normal structure (URNS). This property was  useful  to cover the case  of $L^\infty$-spaces and, more generally, AM-spaces  where the standard normal structure or uniform normal structure do not generally work, in particular when  complex Banach lattices  are considered. The interested reader can find more details in \cite{S}.  
 For the action of groups, the concept of uniform relative normal structure was later used by A. To-Ming Lau  \cite[Theorem 1]{TLau}
 to obtain a common fixed point  for (onto) isometries  defined on a  closed convex bounded subset of a Banach space.  
In \cite[Chapter 5]{KK}, the uniform relative normal structure is defined for the  general environment of metric spaces, with only intersections of balls being considered for its definition. 
In this section we  prove that this metric extension  indeed implies the existence of fixed points for orbit-nonexpansive mappings.  In fact, we will extend Soardi Theorem in \cite{S},   its metric version given in  \cite[Theorem 5.6]{KK} and its corresponding generalization obtained in \cite[Section 4]{EJS} for a single mapping and for the action of groups in the orbit-nonexpansive setting. Notice that some arguments have to be redefined due to the lack of continuity assumptions.   We recall the metric definition of URNS that can be found  in \cite{KK} and the corresponding theorem proved there:

\begin{definition}\label{RUNS} A metric space $\left(M,d\right)$ is said to have  uniform relative normal structure (URNS) if there exists some $c\in\left(0,1\right)$ such that, for every admissible set $A$ with $\delta\left(A\right)>0$, the following conditions are satisfied:

\begin{itemize}
\itemsep=-0.2em
\item [i)] There exists $z_A\in M$ such that 
$
D(z_A,A)\leq c \ \delta\left(A\right).
 $

\item [ii)] If $x\in M$ is such that $D(x,A)\leq c \ \delta\left(A\right)$ then 
$
d\left(x,z_A\right)\leq c \ \delta\left(A\right).
$

\end{itemize}

\end{definition}

\begin{theorem}\label{RUNSFPP} \cite[Theorem 5.6]{KK}, \cite{S} Let $\left(M,d\right)$ be a bounded metric space with URNS and such that
$\mathcal{A}\left(M\right)$ is compact.
Then every nonexpansive mapping $T:M\to M$ has a fixed point. 

\end{theorem}


For a better handling of the uniform relative normal structure, we will use 
 the notation introduced in \cite{EJS}: 
 For $A\subset M$ and $r>0$, let us denote 
\begin{center}
$B\left[A,r\right]:=\displaystyle\bigcap_{x\in A}B\left(x,r\right)=\left\{y\in M:A\subset B\left(y,r\right)\right\}$.
\end{center}
The set $B\left[A,r\right]$ is admissible and  $y\in B[A,r]$ if and only if $A\subset B(y,r)$  if and only if $D(y,A)\le r$. Besides $\delta(A\cap B[A,r])\leq r$ (although it can be an empty set). 
Hence, condition i) in Definition \ref{RUNS} can be equivalently written  as 
$$
B[A, c\ \delta(A)]\ne \emptyset.
$$
Finally, conditions i) and ii)  in Definition \ref{RUNS} can be expressed as a unique statement given by
\[
 \hspace{0.05cm} B\left[A,c \ \delta\left(A\right)\right]\cap B\left[B[A,c \ \delta\left(A\right)],c \ \delta\left(A\right)\right]\neq\emptyset.
\]
for every $A\in\mathcal{A}(M)$ with $\delta(A)>0$. 


There is a subtle but important difference between Definition \ref{RUNS} and the concepts of metric normal structure and  UNS introduced in Section 2 . Note that UNS  can be reformulated by the existence of some $c\in(0,1)$ such that
$$
A\cap B[A, c\ \delta(A)]\ne \emptyset
$$
for all admissible set $A$ with positive diameter.
The point $z_A$ given in Definition \ref{RUNS}  belongs to  $B[A, c\ \delta(A)]$ but it needs not belong to the set $A$ and that is why the extra condition ii) is required. In particular, if a metric space $\left(M,d\right)$ has uniform normal structure then it has uniform relative normal structure by using the same parameter.

A new property that formally weakens URNS was introduced in \cite{EJS} and it was used to prove that, for some metric spaces, the original metric can  be slightly altered such that the following property is still preserved.

\begin{definition}\cite{EJS}\label{pqRUNS} A metric space $\left(M,d\right)$ is said to have $\left(p,q\right)$-URNS for some $p>0$ and $q\in\left(0,1\right)$ if

\begin{center}
$B\left[A,p \ \delta\left(A\right)\right]\cap B\left[B\left[A,p \ \delta\left(A\right)\right],q \ \delta\left(A\right)\right]\neq\emptyset$ 

\end{center}

\noindent for every $A\in\mathcal{A}\left(M\right)$ with $\delta(A)>0$.

\end{definition}

\begin{remark}
Observe that when $p=q\in\left(0,1\right)$, Definition \ref{pqRUNS} is just Definition \ref{RUNS} for the parameter $c:=q$. Hence, $\left(p,q\right)$-URNS provides a formal extension of  URNS. The advantage of considering $(p,q)$-URNS lies in the fact that, playing with the parameters $p$ and $q$, it is possible to prove that for some metric spaces,  as it is the case of hyperconvex spaces,
  the $(p,q)$-URNS is stable when the hyperconvex metric is slightly altered to give place to a second equivalent distance not too far enough from the original one (see \cite[Section 4]{EJS} for a complete exposition of this fact).
\end{remark}

We next extend the main theorem in \cite[Section 4]{EJS} to orbit-nonexpansive mappings. The following notation will be needed: Given $(M,d)$ a metric space and $\mathcal{F}$ a family of self-mappings defined on $M$, the $\mathcal{F}$-{\it admissible} cover of a set $E\subset M$ is defined as:
$$
{\rm cov}_{\mathcal{F}}(E)=\bigcap\left\{A\in\mathcal{A}(M): E\subset A, T(A)\subset A \ \text{for all } T\in\mathcal{F}\right\}.
$$

It is immediate to see that ${\rm cov}_{\mathcal{F}}(E)$ is admissible,  contains $E$ and is $T$-invariant for all $T\in\mathcal{F}$. Once more, the main theorem of this section will be proved for the case of a general interlaced orbit-nonexpansive family  and we will obtain, as  particular cases, a fixed point result for a single orbit-nonexpansive mapping and a common fixed point result for the action of a group.

\begin{theorem}\label{PQRUNSFPP} Let $\left(M,d\right)$ be a bounded metric space with $(p,q)$-URNS for some $0<q<1$ and such that $\mathcal{A}\left(M\right)$ is compact. Let $\mathcal F$ be a family of interlaced orbit-nonexpansive self-mappings on $M$. There exists a common fixed point for all mappings in $\mathcal F$.
\end{theorem}

\begin{proof} Following similar arguments as in \cite{KK} (see also \cite{EJS}), we pursue to construct 
 a sequence of subsets $\left(A_n\right)_{n\ge 0}$ in $\mathcal{A}\left(M\right)$ which are $T$-invariant for all $T\in \mathcal{F}$ and  fulfilling the following properties:

\begin{itemize}
	\itemsep=-0.2em

	\item[(1)] $A_n\subset B\left[A_{n-1},p \ \delta\left(A_{n-1}\right)\right]$.
	\item[(2)] $\delta\left(A_n\right)\leq q \ \delta\left(A_{n-1}\right)$.
\end{itemize}

Assume that (1) and (2) hold and for each $n\in\mathbb{N}$, choose $x_n\in A_n$. Hence
$$
d(x_n,x_{n-1})\le p\ \delta(A_{n-1})\le p q^n\ \delta(A_0).
$$
Therefore, $(x_n)$ is a Cauchy sequence in a complete metric space (since $\mathcal{A}(M)$ is compact \cite[Proposition 3.2]{EK}). Let $x\in M$ be the limit of $(x_n)$. We cannot derive that $x$ is a fixed point from continuity assumptions as in \cite[Theorem 4.7]{EJS}.

 Fix any $T\in\mathcal{F}$. The $T$-invariance of the sets $A_n$ for all $n\ge 0$, lets us assure that 
$$
D(x_n,o_T(x_n))\le \delta(A_n)
$$
and
$$
D(x,o_T(x_n))\le d(x,x_n)+D(x_n,o_T(x_n))\le d(x,x_n)+\delta(A_n).
$$

Using that $\mathcal F$ is a family of interlaced nonexpansive self-mappings w.r.t. orbits,  for all $S\in\mathcal{F}$ we obtain:
$$
d(Sx,Sx_n)\le \sup\{D(x,o_T(x_n)): T\in\mathcal{F}\}\le d(x,x_n)+\delta(A_n).
$$
In conclusion,
$$
\begin{array}{lll}
d(x,Sx)&\le &d(x,x_n)+d(x_n,Sx_n)+d(Sx,Sx_n)\\
            &\le & d(x,x_n) +\delta(A_n)+d(x,x_n)+\delta(A_n)\to_n 0,
            \end{array}
$$
which implies that $x$ is a fixed point of $S$ for all $S\in\mathcal{F}$.

Therefore, it remains to be shown that we can construct a sequence $(A_n)_{n\ge 0}$ of admissible sets verifying (1) and (2). The proof is inspired by the one given in \cite{EJS}. 

Let $A_0\in \mathcal{A}(M)$ be minimal $T$-invariant for all $T\in\mathcal{F}$. Then $A_0={\rm cov}(\displaystyle\bigcup_{T\in\mathcal{F}}T(A_0))$. Denote $\delta_0:=\delta(A_0)$.

Let $x\in M$ such that $A_0\subset B(x,r)$ for some $r>0$. For  $y\in A_0$, $o_R(y)\subset A_0$ for all $R\in\mathcal{F}$. Besides, for  $S,T\in\mathcal{F}$  
$$
d(Sy,Tx)\le \sup\{ D(x, o_R(y)): R\in\mathcal{F}\}\le D(x,A_0)\le r,
$$
which implies that $S(A_0)\subset B(Tx,r)$ and $A_0={\rm cov}(\displaystyle\bigcup_{S\in\mathcal{F}}S(A_0))\subset B(Tx,r)$. This concludes that 
$B[A_0, p \ \delta_0]$ is a $T$-invariant admissible set for all $T\in\mathcal{F}$ (by $(p,q)$-URNS we can set $r:=p \ \delta_0$ in the previous argument).

Define
\begin{equation}\label{a}
\tilde{A}_0:={\rm cov}_{\mathcal{F}}\left(B[A_0, p\ \delta_0]\cap B\left[ B[A_0, p\ \delta_0], q\ \delta_0\right]\right).
\end{equation}

which is nonempty since $(M,d)$ has $(p,q)$-URNS.

First notice that 
\begin{equation}\label{b}
B[A_0, p\ \delta_0]\cap B[ B[A_0, p\ \delta_0], q \ \delta_0]\subset \tilde{A_0}\subset B[A_0, p \ \delta_0]
\end{equation}
where the last inclusion follows from the fact that $B[A_0, p\ \delta_0]$ is $T$-invariant for all $T\in\mathcal{F}$. Additionally, the above implies that
\begin{equation}\label{bc}
B\left[B[A_0, p \ \delta_0],q\ \delta_0\right]\subset B[\tilde{A_0}, q\ \delta_0].
\end{equation}

\medskip 

We next claim that $\delta(\tilde{A_0})\le q\ \delta_0$, that is, $\tilde{A_0}\subset \tilde{A_0}\cap B[\tilde{A_0}, q\ \delta_0]$. To prove the claim, from (\ref{b}) and (\ref{bc}) it can easily be checked that 
$$
B[A_0, p\ \delta_0]\cap B[ B[A_0, p\ \delta_0], q \delta_0]\subset \tilde{A_0}\cap B[\tilde{A_0}, q\ \delta_0].
$$
If we prove that $\tilde{A_0}\cap B[\tilde{A_0}, q\ \delta_0]$ is $T$-invariant for all $T\in\mathcal{F}$, the claim immediately follows. Let $T\in\mathcal{F}$ and 
$$
x\in\tilde{A_0} \cap B[\tilde{A_0}, q\ \delta_0],
$$ 
which implies that $Tx\in \tilde{A_0}$. We want to prove that $Tx\in B[\tilde{A_0}, q\ \delta_0]$, which is equivalent to proving  $\tilde{A_0}\subset B(Tx, q\ \delta_0)\cap \tilde{A_0}$. Notice that 
$$
d(y, Tx)\le q\ \delta_0 
$$
for all $y\in B[A_0, p\ \delta_0]\cap B[ B[A_0, p\ \delta_0], q\ \delta_0]$, by $(\ref{b})$. Hence
$$
B[A_0, p\ \delta_0]\cap B[ B[A_0, p \ \delta_0], q\ \delta_0]\subset B(Tx, q\  \delta_0)\cap \tilde{A}_0.
$$
The proof will finish after checking that $ B(Tx, q\  \delta_0)\cap \tilde{A}_0$ is $S$ invariant for all $S\in\mathcal{F}$.

Let $z\in B(Tx,q\ \delta_0)\cap \tilde{A_0}$. In particular $o_R(z)\subset \tilde{A}_0$ for all $R\in\mathcal{F}$. Additionally, for any $S\in\mathcal{F}$ we have that, 
$$
d(Tx,Sz)\le \sup\{D(x, o_R(z)): R\in\mathcal{F}\}\le D(x,\tilde{A_0})\le  q\ \delta_0,  
$$
since $x\in B[\tilde{A_0}, q\ \delta_0]$.

The above implies that $B(Tx,q\ \delta_0)\cap \tilde{A_0}$ is $S$-invariant for all $S\in\mathcal{F}$ and therefore that $\tilde{A}_0\subset B(Tx,q\ \delta_0)\cap \tilde{A}_0$. In particular $Tx\in \tilde{A}_0\cap B[\tilde{A}_0,q\ \delta_0]$ and consequently the claim is proved. 


Thus, we have found an admissible subset $\tilde{A}_0\subset B[A_0,p\ \delta_0]$  which is $T$-invariant for all $T\in\mathcal{F}$ and such that $\diam(\tilde{A}_0)\le q\ \delta_0$. We proceed by using Zorn's Lemma and find $A_1\subset \tilde{A_0}$ an admissible set which is  minimal and  $T$-invariant for all $T\in\mathcal{F}$. Following the proof by recursion we construct the sequence $(A_n)_{n\ge 0}$ as wished and so the proof is complete. 
 
 \end{proof}

\begin{corollary}
Let $\left(M,d\right)$ be a bounded metric space with uniform relative normal structure and such that $\mathcal{A}\left(M\right)$ is compact. Let $T:M\to M$ be an orbit-nonexpansive mapping. Then $T$ has a fixed point. 
\end{corollary}

\begin{corollary}Let $\left(M,d\right)$ be a bounded metric space with  uniform relative normal structure and such that $\mathcal{A}\left(M\right)$ is compact.  Let $\mathcal{S}$ be a group of orbit-nonexpansive self-mappings  on $M$. Then there is some $x\in M$ such that $s(x)=x$ for all $s\in\mathcal{S}$.
\end{corollary}

We would like to finish the article with the following remark. The compactness of the family of admissible sets has been  assumed along the paper to obtain a minimal invariant set. In many occasions this assumption is given by the intrinsic conditions of the metric space. For instance, if the metric space is endowed with a topology $\tau$ for which the closed balls are $\tau$-closed, the compactness of $\mathcal{A}(M)$ is straightforward whenever $M$ is $\tau$-compact. In particular,  this holds for 
 weak compact or weak$^*$-compact domains  in Banach spaces. Additionally, as it was mentioned before,  the uniform normal structure implies compactness of $\mathcal{A}(M)$ (see \cite{Kh0} or \cite[Theorem 5.4]{KK}) for complete metric spaces. Whether the compactness of the family of admissible sets can be deduced  from the  URNS in complete metric spaces is still an open problem.

\end{document}